\documentclass[a4paper,10pt]{article}
\usepackage{cite}
\usepackage{amsmath}
\usepackage{amsthm}
\usepackage{amsfonts}
\usepackage{amssymb}
\usepackage{color}
\usepackage{tikz}
\usepackage{hyperref}

\usepackage{titlesec}
\titleformat*{\subsection}{\bfseries}

\newcommand{\Ex}{\mathbb{E}}
\renewcommand{\Pr}{\mathbb{P}}
\newcommand{\cond}{{ \; \big\vert \; }}
\newcommand{\prob}[1]{{\Pr\bigl[{#1}\bigr]}}
\newcommand{\avg}[1]{{\Ex\bigl[{#1}\bigr]}}

\renewcommand{\bar}{\overline}

\newtheorem{theorem}{Theorem} 
\newtheorem{claim}{Claim}
\newtheorem{lemma}[theorem]{Lemma}
\newtheorem{conjecture}[theorem]{Conjecture}
\newtheorem{corollary}[theorem]{Corollary}

\newtheorem{proposition}[theorem]{Proposition}

\newcommand{\cB}{\mathcal{B}}
\newcommand{\cC}{\mathcal{C}}
\newcommand{\cE}{\mathcal{E}}
\newcommand{\cR}{\mathcal{R}}
\newcommand{\cS}{\mathcal{S}}

\newcommand{\eps}{\varepsilon}

\def\mathyp{{\hbox{-}}}

\def\constSh{{560}}
\def\expShLLL{{3307 \over 15996}}
\def\constShRB{{1512}}

\title{\vspace{-0.8cm} Properly colored and rainbow copies of graphs with few cherries\thanks{Research supported by SNSF grant 200021-149111.}}
\author{
Benny Sudakov\thanks{Department of Mathematics, ETH, 8092 Zurich. Email: {\tt benjamin.sudakov@math.ethz.ch}}
\and
Jan Volec\thanks{Department of Mathematics, ETH, 8092 Zurich. Email: {\tt jan@ucw.cz}}
}
\date{}
\begin{document}
\maketitle
\begin{abstract}
Let $G$ be an $n$-vertex graph that contains linearly many cherries (i.e., paths on
$3$ vertices), and let $c$ be a coloring of the edges of the complete graph $K_n$ 
such that at each vertex every color appears only constantly many times.
In 1979, Shearer conjectured that such a coloring $c$ must
contain a properly colored copy of $G$.
We establish this conjecture in a strong form, showing that it holds even for graphs $G$ with $O(n^{4/3})$
cherries and moreover this bound on the number of cherries is best possible up to a constant factor.
We also prove that one can find a rainbow copy of such $G$ in every edge-coloring of $K_n$ in which
all colors appear bounded number of times.

Our proofs combine a framework of Lu and Sz\'ekely for using the lopsided Lov\'asz local lemma
in the space of random bijections together with some additional ideas.
\end{abstract}

\section{Introduction}
\label{sec:intro}

The canonical version of Ramsey's theorem~\cite{bib:erdosrado} for graphs implies that for every
graph $G$, there exists an integer $n$ such that any coloring of the edges of the
complete graph $K_n$ contains at least one of the following copies of $G$:
\begin{itemize}
\item \emph{a monochromatic} copy, i.e., a copy where all the edges have the same color,
\item \emph{a rainbow} copy, which is a copy where no two edges have the same color,~or
\item \emph{a lexicographic} copy, in which case the vertices of the copy can be ordered
in such a~way that the color of any edge is purely determined by the smaller endpoint.
\end{itemize}
Note that by restricting the number of colors that the coloring of $E(K_n)$
can use to $k$, the theorem guarantees a monochromatic copy of $K_\ell$
for any fixed $\ell > k$, which implies the classical Ramsey's theorem.

In this paper we consider the following two different types of restrictions, which are kind of dual
to bounding the number of colors: we do not allow any color to, either locally or globally,
appear too many times. More precisely, we say that a coloring $c$ of $E(K_n)$ is \emph{locally $k$-bounded}
if for every vertex $v \in V(K_n)$, no color appears more than $k$-times on the
edges incident to $v$.  Analogously, we say that $c$ is \emph{globally $k$-bounded} if 
no color appears more than $k$-times on all the edges of $K_n$.
We define that a coloring $c$ of $E(K_n)$ is \emph{$G$-proper}, 
if there exists a copy of $G$ in $K_n$ for which $c$ induces a 
\emph{proper edge-coloring}, i.e., a coloring where no two incident edges have
the same color. Similarly, we say that $c$ is \emph{$G$-rainbow}
if there exists a copy of $G$ in $K_n$ such that no two edges of this copy
have the same color in $c$. Given a graph $G$, we would like to obtain sufficient conditions
on an edge-coloring of $K_n$ which yield either a properly colored or a rainbow
copy of this graph.  This problem was studied extensively by various researchers
in the last forty years.

\subsection{\normalsize Locally bounded colorings and properly colored subgraphs}
\label{sec:intro-local}
A conjecture of Bollob\'as and Erd\H{o}s~\cite{bib:bolloberdos} from 1976 states that every
locally $(n/2)$-bounded coloring of $E(K_n)$ is $C_n$-proper, i.e., it contains
a properly colored Hamilton cycle. In~\cite{bib:bolloberdos}, they proved
a weaker result -- any locally $\alpha n$-bounded coloring is $C_n$-proper, where the constant $\alpha$ equals to $1/69$.
Around the same time, Chen and Daykin~\cite{bib:chendayk} showed that already
$\alpha=1/17$ is enough. Then in 1979, Shearer~\cite{bib:shearer} improved the value of $\alpha$ to $1/7$.
After another improvement due to Alon and Gutin~\cite{bib:alongutin}, 
Lo~\cite{bib:lo} proved the conjecture of Bollob\'as and Erd\H{o}s asymptotically. He showed that locally $\alpha n$-bounded colorings are 
$C_n$-proper for any
$\alpha<1/2$ and sufficiently large $n$.

Thirty five years ago, Shearer ~\cite{bib:shearer} proposed the following generalization of the
conjecture above to an arbitrary graph $G$ that does not contain too many \emph{cherries}, i.e.,
paths on three vertices.
\begin{conjecture}
\label{conj:shearer}
For every two integers $s$ and $k$, there exists an integer $n_0$ such that the following is true.
If $n\ge n_0$ and $G$ is an $n$-vertex graph with at most $sn$ cherries, then any
locally $k$-bounded coloring of $E(K_n)$ is $G$-proper.
\end{conjecture}

We establish this conjecture in a strong form, showing that it holds even for graphs $G$ with $O(n^{4/3})$
cherries.
\begin{theorem}
\label{thm:shearer}
If $G$ is an $n$-vertex graph with at most $r$ cherries, then any
locally $\left({\frac{n}{\constSh r^{3/4}}}\right)$-bounded coloring $c$ of $E(K_n)$ is
$G$-proper.
\end{theorem}
This result is tight up to a constant factor. In Section~\ref{sec:lbounds}, we will construct locally $3$-bounded colorings $c_n$ of $E(K_n)$ together
with $n$-vertex trees $T_n$ with $\Theta(n^{4/3})$ cherries so that $c_n$ is not $T_n$-proper.

Another generalization of the conjecture of Bollob\'as and Erd\H{o}s to a general
graph $G$ takes into account the maximum degree.
Alon, Jiang, Miller and Pritikin~\cite{bib:alonjiang} showed that if 
$G$ is an $n$-vertex graph with maximum degree $\Delta$ and $k=O\left(\frac{\sqrt n}{\Delta^{27/2}}\right)$,
then any locally $k$-bounded coloring $c$ of $E(K_n)$ is $G$-proper.
Their result was greatly improved by B\"ottcher, Kohayakawa and Procacci~\cite{bib:bkp} who showed
that $k$ can be of order $n/\Delta^2$.
\begin{theorem}
\label{thm:bkp:1}
If $G$ is an $n$-vertex graph with maximum degree $\Delta$,
then any locally $\left({{n}/{22.4\Delta^2}}\right)$-bounded
coloring $c$ of $E(K_n)$ is $G$-proper.
\end{theorem}
Can one further improve this bound?  Our next contribution shows that up to a constant factor,
this result is tight for all values $n$ and $\Delta$.  Moreover, one can find graphs $G$
with maximum degree $\Delta$ and locally $(3.9n/\Delta^2)$-bounded but not $G$-proper colorings,
of $K_n$, where the number of vertices of $G$ does not depend on $n$ at all.
\begin{proposition}
\label{prop:constr1}
For every prime power $q$ and integer $n$, 
there exist an $\ell$-vertex graph $G$ with maximum degree $\Delta$, where $\ell=q^2+q+1$ and $\Delta=q+1$,
and a locally $(3.9n/\Delta^2)$-bounded coloring $c$ of $E(K_n)$ so that $c$ is not $G$-proper.
\end{proposition}

\subsection{\normalsize Globally bounded colorings and rainbow subgraphs}
\label{sec:intro-global}
There is a rich literature studying rainbow copies of a fixed graph
in globally bounded colorings of $E(K_n)$,
see for example~\cite{bib:afr,bib:aghh,bib:friezekriv,bib:galvin,bib:hahnthomassen,bib:harveyvondrak,bib:hellmb,bib:lrw}.
In this work, we will focus on finding rainbow spanning subgraphs.

Various authors have considered an analogue of the Bollob\'as-Erd\H{o}s conjecture,
where the aim is to find a rainbow Hamilton cycle in a globally bounded coloring of $E(K_n)$.
Specifically, in 1986 Hahn and Thomassen~\cite{bib:hahnthomassen} conjectured that 
there is a constant $\alpha > 0$ such that any globally $\alpha n$-bounded coloring of
$K_n$ is $C_n$-rainbow. Their conjecture was proven by Albert, Frieze, and
Reed~\cite{bib:afr} with $\alpha=1/64$ (see also~\cite{bib:rue} for a correction of
the originally claimed constant).

In 2008, Frieze and Krivelevich~\cite{bib:friezekriv} showed that there is some
absolute constant $\alpha > 0$ so that any globally $\alpha n$-bounded coloring actually
contains copies of $C_k$ for all $k \in \{3,\dots,n\}$. In the same paper, they conjectured
that there is also a constant $\alpha > 0$ such that every globally $\alpha n$-bounded coloring
contains any spanning tree with bounded maximum degree.
Using the same technique as for proving Theorem~\ref{thm:bkp:1},
B\"ottcher, Kohayakawa and Procacci~\cite{bib:bkp}
proved the conjecture of Frieze and Krivelevich not only for trees, but actually for all spanning
subgraphs with bounded maximum degree.
\begin{theorem}[\cite{bib:bkp}]
\label{thm:bkp:2}
If $G$ is an $n$-vertex graph with maximum degree $\Delta$,
then
any globally $\left({{n}/{51\Delta^2}}\right)$-bounded
coloring $c$ of $E(K_n)$ is $G$-rainbow.
Furthermore, if $n\ge100$,
then any globally $\left({{n}/{42\Delta^2}}\right)$-bounded
coloring $c$ of $E(K_n)$ is $G$-rainbow.
\end{theorem}
With a slight modification of the construction from
Proposition~\ref{prop:constr1}, we can show that the dependency
$k=O(n/\Delta^2)$ in Theorem~\ref{thm:bkp:2} is again best possible.
\begin{proposition}
\label{prop:constr2}
For every two integers $\Delta$ and $n$ such that $\Delta$ is even and $\left(\frac\Delta2+1\right)^2$ divides $n$, 
there exist an $n$-vertex graph $G$ with maximum degree $\Delta$
and a globally $(16n/\Delta^2)$-bounded coloring $c$ of $E(K_n)$ so that $c$ is not $G$-rainbow.
\end{proposition}

Finally, one can naturally ask what can be said about rainbow copies of graphs with few cherries
in globally bounded edge-colorings of $K_n$.
We were able to answer this question as well, proving the following analog
of Conjecture~\ref{conj:shearer} in this setting.
\begin{theorem}
\label{thm:rbshearer}
If $G$ is an $n$-vertex graph with at most $r$ cherries, then any
globally $\left({\frac{n}{\constShRB r^{3/4}}}\right)$-bounded coloring $c$ of $E(K_n)$
is $G$-rainbow.
\end{theorem}
Since the locally $3$-bounded coloring $c$ of $E(K_n)$ which shows 
the tightness of Theorem~\ref{thm:shearer} is also globally $9$-bounded, 
we conclude that again the number of cherries cannot exceed $\Theta(n^{4/3})$.


\section{Local lemma in the space of random bijections}
\label{sec:lll}
The Lov\'asz local lemma is a tool used for showing the existence of an object 
that does not possess any property from a given list of unwanted properties.
This is achieved by taking a random object and showing that with a positive probability,
the object has none of the unwanted properties. In order to be able
to apply the local lemma, we need to have some control over the mutual correlations
of these properties.

Let $\cB=\{B_1,\dots,B_N\}$ be a set of events, where each event describes
having one of the unwanted properties. The events are usually called the bad
events. We say that a graph $D$ with the vertex set $[N]$ is a \emph{dependency graph}
for $\cB$ if for every $i \in [N]$, the event $B_i$ is mutually independent of
all the events $B_j$ such that $ij \notin E(D)$. In other words, for every 
$i \in [N]$ and every set $J \subseteq \{j : ij \notin E(D)\}$,
it holds that $\prob{B_i \cond \bigwedge_{j \in J} \bar{B_j}\;} = \prob{B_i}$.
Analogously, we say that an $N$-vertex graph $D$ is a \emph{negative dependency graph}
for $\cB$ if for every $i \in [N]$ and every set~$J \subseteq \{j : ij \notin E(D)\}$,
it holds that $\prob{B_i \cond \bigwedge_{j \in J} \bar{B_j}\;} \le \prob{B_i}$.

The original version of the local lemma, which is due to Erd\H{o}s and Lov\'asz~\cite{bib:lll}, 
used a dependency graph for the set of \emph{bad events} in order to control the correlations.
It was first observed by Erd\H{o}s and Spencer~\cite{bib:llll} that actually
the same proof also applies when we capture the correlations 
using a negative dependency graph.
They called this variant \emph{lopsided Lov\'asz local lemma}.
The following is a slightly more general
version of the lemma than the one stated in~\cite{bib:llll}, whose
proof can be found, e.g., in~\cite[Lemma 1.4]{bib:mohr}.
\begin{lemma}[Lopsided Lov\'asz local lemma]
\label{lem:LLLL} 
Let $\cB=\{B_1,\dots,B_N\}$ be a set of bad events with a negative dependency graph $D=([N], \cE)$.
If there exist reals $b_1,\dots,b_N \in (0,1)$ so that
\[\prob{B_i} \le b_i \cdot \prod_{ij \in \cE} (1-b_j) \qquad \textrm{for every }i\in [N], \]
then $\prob{\bigwedge\limits_{i\in[N]} \bar{B_i}\; } > 0$.
\end{lemma}

In our applications, we will be only using the following simpler version of the local lemma,
which is in fact an easy corollary of Lemma~\ref{lem:LLLL}.
Note that this version is often called the asymmetric local lemma (see, e.g.,~\cite[Chapter 19]{bib:molloyreed}):
\begin{lemma}
\label{lem:wLLL} 
Let $\cB=\{B_1,\dots,B_N\}$ be a set of bad events with a negative dependency graph $D=([N], \cE)$.
If
\[\prob{B_i} \le \frac14 \quad\textrm{and}\quad \sum_{ij \in \cE}\prob{B_j}\le \frac14 \qquad \textrm{for every }i\in [N], \]
then $\prob{\bigwedge\limits_{i \in [N]} \bar{B_i}\; } > 0$.
\end{lemma}
The lopsided variant of the asymmetric local lemma is 
mentioned in~\cite[Chapter 19.4]{bib:molloyreed} only implicitly.
However, its proof is identical to the proof where $D$ is only a dependency graph,
which is proven in~\cite[Chapter 19.3]{bib:molloyreed}.

The most important thing in many applications of the (lopsided) local lemma
is to find an appropriate (negative) dependency graph for a given set of bad events. 
Lu and Sz\'ekely~\cite{bib:luszek} came up with a particularly useful
construction of a negative dependency graph in the case that the underlying
probability space is generated by taking a random bijection between two sets.

Let $X$ and $Y$ be two sets of size $n$ and $\cS_n$ the set of all bijections from $X$ to $Y$. 
Consider the probability space $\Omega$ 
generated by picking a uniformly random element of $\cS_n$. 
We say that an event $B$ is \emph{canonical} if there exist two sets $X'\subseteq X$, $Y'\subseteq Y$ and a
bijection $\tau: X' \to Y'$ such that $B = \{\pi \in \cS_n : \pi(a) = \tau(a) \textrm{ for all } a \in X'\}$.
For two sets $X' \subseteq X$ and $Y' \subseteq Y$ of the same size and a bijection $\tau: X' \to Y'$,
we denote the corresponding canonical event by $\Omega(X',Y',\tau)$.

We say that two events $\Omega(X'_1,Y'_1,\tau_1)$ and $\Omega(X'_2,Y'_2,\tau_2)$ \emph{$\cS$-intersect} if 
the sets $X'_1$ and $X'_2$ intersect, or the sets $Y'_1$ and $Y'_2$ intersect.
A result of Lu and Sz\'ekely~\cite{bib:luszek} states that
for a set of bad canonical events, the graph with vertices
being the bad events and edges being between any two events that
$\cS$-intersect is a negative dependency graph.
\begin{theorem}[\cite{bib:luszek}]
\label{thm:luszekely}
Let $\Omega$ be the probability space generated by picking a random bijection between two sets $X$ and $Y$ of size $n$
uniformly at random.
Next, let $\cB=\{B_1,\dots,B_N\}$ be a set of canonical events in $\Omega$ and
let $D$ be a graph with the vertex set $[N]$ and $ij \in E(D)$ if and only if
the events $B_i$ and $B_j$ $\cS$-intersect. It holds that $D$ is a negative dependency graph.
\end{theorem}
Let us note that Lu and Sz\'ekely~\cite{bib:luszek} proved the statement above with a slightly
better choice of the negative dependency graph. Namely, they showed that a
graph $D'$ with the set of vertices $[N]$, where a vertex representing $\Omega(X'_1,Y'_1,\tau_1)$ is adjacent to 
a vertex representing $\Omega(X'_2,Y'_2,\tau_2)$ if and only if 
\[
\big(\exists x \in X'_1 \cap X'_2: \tau_1(x)\neq\tau_2(x)\big)
\,\textrm{ or }\,
\big(\exists y \in Y'_1 \cap Y'_2: \tau_1^{-1}(y)\neq\tau_2^{-1}(y)\big)
,\] is a negative dependency graph.
In other words, $\Omega(X'_1,Y'_1,\tau_1)$ and $\Omega(X'_2,Y'_2,\tau_2)$ 
are adjacent in $D'$ if and only if the two probability events in $\Omega$ are disjoint.
It immediately follows that $D'$ is a subgraph of $D$, and since $D'$ is a negative dependency graph,
the graph $D$ must be a negative dependency graph as well. 

\section{Proofs of Theorems~\ref{thm:shearer} and~\ref{thm:rbshearer}}
\label{sec:proof}
Before we start with a rigorous proof, let us give a brief outline.
As we have seen in the introduction, the local lemma is the right tool if the
maximum degree $\Delta(G) = O\left(\sqrt{n}\right)$.
Unfortunately, our upper bound on the number of
cherries 
cannot provide such a strong control on $\Delta(G)$.  However, a straightforward counting argument
yields that only a very small number of vertices in $G$ can have a degree of order $\Omega\left(\sqrt{n}\right)$.
Furthermore, we show in Lemma~\ref{lem:step1} that since $c$ is locally (globally) bounded,
there is a complete subgraph $H$ of $K_n$ of the appropriate size that is properly colored (rainbow) in
$c$, and also no two of its vertices have too large monochromatic co-degree in
$V(K_n)\setminus V(H)$.  Therefore, we can map the large-degree vertices of $G$
to the vertices of $H$, and map the other vertices of $G$ using the local lemma.
In order to get strong bounds, we will also need precise upper bounds on
the number of edges of $G$ of certain types, and on the number of paths of
length $2$ starting at a given vertex. Those bounds are established in
Lemmas~\ref{lem:edges} and~\ref{lem:leafcher}, respectively, using the Cauchy-Schwarz
inequality.

Through the whole section, we will omit floors and ceilings whenever it is not critical. 
We start our exposition with the following three auxiliary lemmas.
\begin{lemma}
\label{lem:step1}
For all positive integers $n$, $k$ and $r$ such that 
$k\le \left({\frac{n}{\constSh r^{3/4}}}\right)$,
the following is true.
Every locally (globally) $k$-bounded coloring $c$ of $K_n$ contains a~properly colored 
(rainbow) complete subgraph $H$ of size $2r^{1/4}$ such that for every
two vertices $v_1,v_3 \in V(H)$, the set $\{v_2 \in V(K_n) : c(v_1 v_2) = c(v_2 v_3) \}$ has
size at~most~$5kr^{1/4}$.
\end{lemma}
\begin{proof}
First note that (both locally and globally) $k$-bounded colorings contain at most
$\frac12n(n-1)k$ monochromatic paths on three vertices.  To see that, we claim
that for a fixed choice of the middle vertex $v_2$ of such a path, there are at
most $\frac12(n-1)k$ choices for the two endpoints of the path. Indeed, after
choosing one of the endpoints, which can be done in $(n-1)$ ways, there are at
most $k$ possible other endpoints so that the path monochromatic. Furthermore,
we counted every monochromatic path with $v_2$ as the middle point exactly
twice. Summing over all choices of $v_2$ yields the bound $\frac12n(n-1)k$.

Now let $A$ be the following auxiliary graph: the vertex set is $V(K_n) = [n]$, and the
vertices $v_1 \in V(A)$ and $v_3\in V(A)$ are adjacent if and only if there exist at least
$5kr^{1/4}$ vertices $v_2\in$ so that $c(v_1v_2) = c(v_2v_3)$. It follows that the number
of edges of $A$ is at most 
$\frac{n(n-1)}{10r^{1/4}}$. We denote the number of edges of $A$ by $e(A)$.

We construct the desired subgraph $H$ using the first moment method.
Let $p := 5r^{1/4} \cdot n^{-1}$, and let $P'$ be a random subset of $[n]$ where
we put each element with probability $p$ independently on the others. 
The expected size of $P'$ is $5r^{1/4}$, and the expected number of
edges of the subgraph of $A$ induced by $P'$ is at most $e(A) \cdot p^2 \le 2.5r^{1/4}$.
We set $U_1 \subseteq P'$ to be the set containing the smaller of the two vertices
for each edge of the subgraph. It can be that $U_1$ contains both endpoints for some edge
because its larger endpoint is the smaller endpoint of some other edge. 
Note that $\avg{|U_1|} \le 2.5r^{1/4}$, and that
for any two vertices $v_1$ and $v_3$ from $P' \setminus U_1$,
the set $\{v_2 \in V(K_n) : c(v_1 v_2) = c(v_2 v_3) \}$ has size at most~$5kr^{1/4}$.

Next, let $U_2 \subseteq P'$ be the set containing the smallest vertex from every
$\{v_1,v_2,v_3\}\subseteq P'$ with $c(v_1v_2)=c(v_2v_3)$.
It follows that \[
\avg{|U_2|} \le \frac{n^2k \cdot p^3}2 \le \frac{125r^{3/4}\cdot k}{2n} \le \frac{125}{1120} \le \frac18,
\]
and the coloring induced by $c$ on the subgraph $P' \setminus U_2$
is proper.

Finally, if $c$ is globally $k$-bounded, observe that there are at most $n^2 k/4$ 
sets $\{v_1,v_2,v_3,v_4\} \subseteq [n]$ such that $c(v_1v_2) = c(v_3v_4)$.
Let $U_3$ be the set containing the smallest
vertex from every $\{v_1,v_2,v_3,v_4\} \subseteq P'$ with $c(v_1v_2)=c(v_3v_4)$.
In the case of $c$ being locally $k$-bounded, we set $U_3 := \emptyset$.
It holds that
\[
\avg{|U_3|} \le \frac{n^2k\cdot p^4}{4} \le \frac{625rk}{4n^2} \le \frac{625}{2240} \le \frac{3}{8}.
\]
It follows that in the case $c$ is globally $k$-bounded, the subgraph induced by
$P' \setminus \left(U_2 \cup U_3\right)$ is rainbow in $c$.

By linearity of expectation, the set $P:=P' \setminus \left(U_1 \cup U_2 \cup
U_3\right)$ has expected size at least $5r^{1/4} - 2.5r^{1/4} - 0.5 \ge 2r^{1/4}$.
On the other hand, the subgraph induced by $P$ has all the desired properties.
\end{proof}

\begin{lemma}
\label{lem:edges}
Every $n$-vertex graph $G$ with at most $r$ cherries contains 
at most $\max\left\{{n,\sqrt{rn}}\right\}$ edges.
Furthermore, for any subset $T\subseteq V(G)$, the number of edges with at
least one endpoint in $T$ is at most $\max\left\{{4|T|, 2\sqrt{r |T|}}\right\}$.
\end{lemma}
\begin{proof}
Let $e(G)$ be the number of edges of $G$.
We claim that $4e(G)^2 \le (2r + 2e(G))n$.
Indeed, by the Cauchy-Schwarz inequality 
\[
\left(\sum_{u\in V(G)} \deg(u)\right)^2 \le n \cdot \sum_{u\in V(G)} \deg^2(u).
\]
However, $\sum_u \deg(u) = 2e(G)$ and $\sum_u \deg(u)(\deg(u)-1) = 2r$.
Therefore, if $e(G)\ge n$, then $4e(G)^2 \le n(2r+2e(G)) \le 2rn + 2e(G)^2$.

Analogously for the set $T$, let $e(T,G)$ be the number of edges of $G$ with at
least one endpoint in $T$. Note that
\[\frac12 \cdot \sum_{u\in T} \deg(u) \le e(T,G) \le \sum_{u\in T} \deg(u).\]
Again by Cauchy-Schwarz,
\[e(T,G)^2 \le |T| \cdot \left( \sum_{u\in T} \deg(u)(\deg(u)-1) + \sum_{u\in T} \deg(u) \right)
\le 2r|T| + 2|T|e(T,G)
.\]
Hence if $e(T,G) \ge 4|T|$, then $e(T,G)^2 \le 4r|T|$.
\end{proof}

\begin{lemma}
\label{lem:leafcher}
Let $G$ be an $n$-vertex graph with at most $r$ cherries and $u \in V(G)$ one of its vertices.
Then $G$ contains at most $\sqrt{2r \deg\left(u\right)}$ cherries with $u$ being one of the two leaves.
\end{lemma}
\begin{proof}
Let $N \subseteq V(G)$ be the set of the neighbors of $u$. 
The number of cherries, where $u$ is one of the leaves, is equal to $\sum_{u' \in N} (\deg(u')-1)$.
As in the proof of the previous lemma,
\[\left(\sum_{u' \in N} (\deg(u')-1)\right)^2 \le |N| \cdot \sum_{u' \in N} \deg(u')(\deg(u')-1) \le 2r \deg(u)
.
\]
\end{proof}


We are now ready to prove Theorem~\ref{thm:shearer}. 
\begin{proof}[Proof of Theorem~\ref{thm:shearer}]
Let $\Delta_G$ be the maximum degree of $G$, $C:=\constSh$ and $k:=\frac{n}{C r^{3/4}}$.
If $n < 2C$ or $r > (n/2C)^{4/3}$, then $k\le 1$ and hence the statement of the
theorem is trivial. For the rest of the proof, we assume $n \ge 2C$ and $r \le (n/2C)^{4/3}$.
We may also assume that $r\ge16$.  Indeed, if $r \le 15$ then the maximum degree of $G$
is at most $6$.
If $\Delta_G=6$, then $G$ must be a disjoint union of
one star with $6$ leaves and a graph on $(n-7)$ vertices with maximum degree
one. Such a graph can be easily embdedded in a greedy fashion.  On the other
hand, if $\Delta_G \le 5$ then the statement directly follows from
Theorem~\ref{thm:bkp:1}.

%
Now observe that $\Delta_G(\Delta_G-1) \le 2r$ as otherwise the vertex of $G$ with the maximum
degree is contained in more than $r$ cherries. Since $r\ge16$, we conclude that
\begin{equation}
 \Delta_G \le \sqrt{2r} + 1 \le 2\sqrt{r}.
\label{eq:maxdeg}
\end{equation}
Without loss of generality, $V(G) = V(K_n) = [n]$, and the vertices of $V(G)$ are in
the descending order according to their degrees (breaking ties arbitrarily). In other words, if $u,v \in V(G)$ and $u < v$,
then $\deg_G(u) \ge \deg_G(v)$.
Let~{$P \subseteq V(K_n)$} be the properly colored complete subgraph of $K_n$ of
size $\ell :=  2r^{1/4}$ given by Lemma~\ref{lem:step1} for $c$, $r$ and $k$. Set
$Q:=V(K_n) \setminus P$. It follows that for every $v_1,v_3 \in P$ there are at
most $5kr^{1/4}$ choices of $v_2 \in Q$ so that $c(v_1v_2) = c(v_2v_3)$.
On the other hand, let $L$ be the set of the first $\ell$ vertices of $G$,
i.e., 
the set of $\ell$ vertices with the largest degrees.
Let $S := V(G) \setminus L$ and let $\Delta_S := \max_{u \in S} \deg_G(u)$.
Note that
\begin{equation}
\Delta_S (\Delta_S-1) \le 2r / \ell = r^{3/4},
\label{eq:maxdegS}
\end{equation}
as otherwise $G$ contains more than $r$ cherries.

Now we describe how we find a properly edge-colored copy of $G$ in $c$.
First, fix an arbitrary bijective map $f_1: L \to P$. Let us emphasize 
that any such $f_1$ will be possible to extend into a properly colored copy of $G$.
The remaining vertices of
$G$, i.e., the vertices from $S$, are mapped by a uniformly chosen random
bijection $f_2: S \to Q$. Finally, let $f:=f_1 \cup f_2$
be the bijection between $V(G)$ and $V(K_n)$ and let $f(G)$ denote the (random)
copy of $G$ in $K_n$ given by $f$. We use Theorem~\ref{thm:luszekely} and Lemma~\ref{lem:wLLL} to show
that, with a positive probability, the copy $f(G)$ is properly colored by $c$ restricted to the edges of $f(G)$.

Before we proceed further, let us introduce some additional notation.
We denote a cherry in $G$ with the middle vertex $u_2$ and the endpoints $u_1,u_3$ such
that $u_1 < u_3$ by $u_1\mathyp u_2\mathyp u_3$. Through the whole paper, we will
write $u_1\mathyp u_2\mathyp u_3$ only in the case when $u_1 < u_3$.
On the other hand, for $v_1, v_2, v_3 \in V(K_n)$, we say that the triple $[v_1 v_2 v_3]$ is 
$c$-monochromatic if $c(v_1 v_2) = c(v_2 v_3)$.
Let us emphasize that in this definition we assume neither $v_1 < v_3$, nor $v_1 > v_3$.

Let $\cR(G)$ be the set of all cherries in $G$,
and let $\cC(c)$ be the
set of all $c$-monochromatic triples $[v_1v_2v_3]$. Note that
$[v_1v_2v_3] \in \cC(c) \iff [v_3v_2v_1] \in \cC(c)$.
Also note that $|\cC(c)| \le n(n-1)k$, since there are $n(n-1)$ choices
of the vertices $v_1$ and $v_2$, and then at most $k$ choices of $v_3$ so
that $c(v_1v_2)=c(v_2v_3)$.
Our aim is to show that the bijection $f$
is such that for every cherry $u_1\mathyp u_2\mathyp u_3 \in \cR(G)$ 
it holds that $[f(u_1)f(u_2)f(u_3)] \notin \cC(c)$.
Since the image of $f_1$ is $P$, which induces a properly colored clique in $c$,
it follows 
that $[f(u_1)f(u_2)f(u_3)] \notin \cC(c)$ for every cherry $u_1\mathyp u_2
\mathyp u_3$ with $\{u_1,u_2,u_3\} \subseteq L$.

For a cherry $u_1\mathyp u_2\mathyp u_3 \in \cR(G)$ with $\{u_1,u_2,u_3\} \cap S \neq \emptyset$
and a triple $[v_1v_2v_3] \in \cC(c)$,
let $B^{[v_1v_2v_3]}_{u_1\mathyp u_2 \mathyp u_3}$ denote the event $\bigwedge\limits_{i\in\{1,2,3\}} [f(u_i)=v_i]$,
and let $\cB$ be the set of all events $B^{[v_1v_2v_3]}_{u_1\mathyp u_2 \mathyp u_3}$ that satisfy
\begin{itemize}
\item $u_1\mathyp u_2 \mathyp u_3 \in \cR(G)$ and $[v_1v_2v_3]\in \cC(c)$,
\item $\left\{u_1,u_2,u_3\right\} \cap S \neq \emptyset$,
\item $\forall i \in \{1,2,3\}: u_i \in S \iff v_i \in Q$, and
\item $\forall i \in \{1,2,3\}: u_i \in L \Longrightarrow f_1(u_i) = v_i$.
\end{itemize}
Note that since for every $B\in\cB$ at least one of the vertices $u_i$, where $i\in\{1,2,3\}$,
is mapped to $v_i$ by the randomly chosen bijection $f_2$, it holds that $\prob{B} \le 1/(n-\ell) \le 1/4$.

It follows that two events
$B^{[v_1v_2v_3]}_{u_1\mathyp u_2 \mathyp u_3}$ and $B^{[v_4v_5v_6]}_{u_4\mathyp u_5 \mathyp u_6}$
$\cS$-intersect if and only if the sets $\{u_1,u_2,u_3\}$ and $\{u_4,u_5,u_6\}$ intersect or the sets $\{v_1,v_2,v_3\}$ and $\{v_4,v_5,v_6\}$
intersect.
Lemma~\ref{lem:wLLL} states that in order to conclude  that the probability $\prob{\bigwedge\limits_{B\in\cB} \bar{B}\; } > 0$,
it is enough to show that 
\begin{equation}
\label{eq:proper:localcond}
 \sum\limits_{\substack{B' \in \cB: \\ B\textrm{ and } B'\\\cS\textrm{-intersect}}} \prob{B'} \le \frac14 \qquad \textrm{for every } B \in \cB
.
\end{equation}
To do so, we split the events $B^{[v_1v_2v_3]}_{u_1\mathyp u_2 \mathyp u_3} \in \cB$ into 
five classes $\cB_1,\dots,\cB_5$ based on how their sets $\{u_1,u_2,u_3\}$ intersect the set $S$:
\begin{itemize}
\item If $\{u_1,u_2,u_3\} \subseteq S$, then $B^{[v_1v_2v_3]}_{u_1\mathyp u_2 \mathyp u_3} \in \cB_1$.
\item If $\{u_2, u_3\} \subseteq S$ and $u_1 \in L$, then $B^{[v_1v_2v_3]}_{u_1\mathyp u_2 \mathyp u_3} \in \cB_2$.
\item If $\{u_1,u_3\} \subseteq S$ and $u_2 \in L$, then $B^{[v_1v_2v_3]}_{u_1\mathyp u_2 \mathyp u_3} \in \cB_3$.
\item If $u_3 \in S$ and $\{u_1,u_2\} \subseteq L$, then $B^{[v_1v_2v_3]}_{u_1\mathyp u_2 \mathyp u_3} \in \cB_4$.
\item If $u_2 \in S$ and $\{u_1,u_3\} \subseteq L$, then $B^{[v_1v_2v_3]}_{u_1\mathyp u_2 \mathyp u_3} \in \cB_5$;
\end{itemize}
see Figure~\ref{fig:events_cherries} for an example for each of the classes.
Note that since $u_1<u_3$ it follows that if $u_3 \in L$ then also $u_1 \in L$. Thus indeed the classes $\cB_1,\dots,\cB_5$ split the set $\cB$.
It holds that
\begin{align*}
& \prob{B}=\frac1{(n-\ell)(n-\ell-1)(n-\ell-2)}\quad & \textrm{for any } & B\in\cB_1,\\
& \prob{B}=\frac1{(n-\ell)(n-\ell-1)}\quad & \textrm{for any } & B\in\cB_2 \cup \cB_3, \textrm{ and}\\
& \prob{B}=\frac1{(n-\ell)}\quad & \textrm{for any } & B\in\cB_4 \cup \cB_5.
\end{align*}

\begin{figure}[htp]
\begin{center}
\medskip
\foreach \n in {1,...,4}{ \includegraphics[scale=0.8,page=\n]{fig-lll_simple}\hfill}
{ \includegraphics[scale=0.8,page=5]{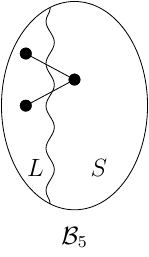}}
\end{center}
\caption{The intersection types defining the classes $\cB_1,\dots,\cB_5$.}
\label{fig:events_cherries}
\end{figure}

For every vertex $u \in S$ and two integers $i \in [5]$ and $j\in[3]$, let $t^{u_j}_i (u)$ be the number of events
$B^{[v_1v_2v_3]}_{u_1\mathyp u_2 \mathyp u_3} \in \cB_i$ such that $u=u_j$.
Note that for every $u \in S$, the values of 
$t^{u_1}_2(u)$, $t^{u_2}_3(u)$, $t^{u_1}_4(u)$, $t^{u_2}_4(u)$, $t^{u_1}_5(u)$ and $t^{u_3}_5(u)$ are equal to~$0$.
Analogously, for every vertex $v \in Q$ and integers $i \in [5]$ and $j\in[3]$, let $t^{v_j}_i (v)$ be the number of events
$B^{[v_1v_2v_3]}_{u_1\mathyp u_2 \mathyp u_3} \in \cB_i$ such that $v=v_j$. In this case,
$t^{v_1}_2(v)$,  $t^{v_2}_3(v)$, $t^{v_1}_4(v)$, $t^{v_2}_4(v)$, $t^{v_1}_5(v)$ and $t^{v_3}_5(v)$
are all zero for every $v \in Q$.
Finally, for every $i \in [5]$, we define
\[t^u_i := \max_{w \in S} \big(t^{u_1}_i(w) + t^{u_2}_i(w) + t^{u_3}_i(w)\big), \]
and
\[t^v_i := \max_{w \in Q} \big(t^{v_1}_i(w) + t^{v_2}_i(w) + t^{v_3}_i(w)\big). \]

For every $B = B^{[v_1v_2v_3]}_{u_1\mathyp u_2 \mathyp u_3} \in \cB$, the set $\{u_1,u_2,u_3\}$ consists
of at most $3$ vertices of $S$. Analogously, $\{v_1,v_2,v_3\}$ consists of at most $3$ vertices of $Q$.
Therefore, 
\[
\sum\limits_{\substack{B' \in \cB: \\ B\textrm{ and } B'\\ \cS\textrm{-intersect}}} \prob{B'} \; \le \;
\sum\limits_{i=1}^5
\prob{B'_i}\cdot3(t^u_i + t^v_i)
%
\mbox{\,,}
\]
where $B'_i \in \cB_i$ for $i \in \{1,\dots,5\}$.

In the following series of claims, we present a careful but most of the time easily
followable calculations, which will lead to bounds on the values of
$t^{u}_i$ and $t^{v}_i$ for $i\in\{1,\dots,5\}$. The bounds from the claims are
summarized in seven corollaries, which we will then put together and 
conclude that the sum above is at most $1/4$.

\begin{claim}
\label{cl1u13}
For every $u\in S$, $t^{u_1}_1(u) + t^{u_3}_1(u) \le \Delta_S (\Delta_S-1) (n-\ell)(n-\ell-1) k$.
\end{claim}
\begin{proof}
Our aim is to upper bound the number of ways how to choose
$u_2, u', v_1, v_2$ and $v_3$ so that $B^{[v_1v_2v_3]}_{u_1\mathyp u_2 \mathyp u_3} \in \cB_1$,
where $u_1:=\min(u,u')$ and $u_3:=\max(u,u')$.
Note that this quantity is exactly equal to $t^{u_1}_1(u) + t^{u_3}_1(u)$.

Firstly, there are at most $\Delta_S$ ways how to choose $u_2\in S$. Once the
vertex $u_2$ is fixed, there are at most $\Delta_S-1$ ways how to choose the
remaining vertex $u' \in S$. 
Next, there are exactly $(n-\ell)(n-\ell-1)$ ways how
to choose the vertices $v_1 \in Q$ and $v_2 \in Q$. Finally, since the color of the edge
$v_2v_3$ should be the same as the color of $v_1v_2$, the vertex $v_3$ can be
chosen in at most $k$ ways.
\end{proof}

\begin{claim}
\label{cl1u2}
For every $u\in S$, $t^{u_2}_1(u) \le \frac12 \Delta_S (\Delta_S-1) (n-\ell)(n-\ell-1) k$.
\end{claim}
\begin{proof}
There are at most $\Delta_S \choose 2$ options for choosing the pair $u_1$ and $u_3$ so that
$u_1u_2 \in E(G), u_2u_3 \in E(G)$ and $u_1 < u_3$.
Next, there are at most $(n-\ell)(n-\ell-1)k$ ways how to choose the vertices $v_1,v_2$ and $v_3$.
\end{proof}

Since $\Delta_S (\Delta_S-1) \le r^{3/4}$ and $k\le\frac n{Cr^{3/4}}$,
we conclude the following.
\begin{corollary}
\label{cor:u1}
For every $B_1 \in \cB_1$,
\[
t^{u}_1  \le \frac {3n}{2C(n-\ell-2)} \cdot \frac1{\prob{B_1}} 
\le \frac{3}{2C-4}\cdot\frac1{\prob{B_1}}
.\]
\end{corollary}
\noindent Note the last inequality follows from the estimates $\ell \le 2(n/2C)^{1/3} \le n/C$ and $2 \le n/C$.

\begin{claim}
\label{cl1v123}
For every vertex $v\in Q$, $t^{v_1}_1(v) + t^{v_2}_1(v) + t^{v_3}_1(v)  \le 3(n-\ell-1)kr$.
\end{claim}
\begin{proof}
We show that each $t^{v_1}_1(v), t^{v_2}_1(v)$ and $t^{v_3}_1(v)$ is at most $(n-\ell-1)kr$.
If $v=v_2$, then there are $(n-\ell-1)$ ways how to choose $v_1$ and at most $k$ ways how to
choose $v_3$. On the other hand, if $v\in\{v_1,v_3\}$, then there are $(n-\ell-1)$ ways how
to choose $v_2$ and then at most $k$ ways how to choose the remaining vertex in $Q$. Finally, 
in all the cases there are at most $r$ choices for a cherry $u_1\mathyp u_2 \mathyp u_3$.
\end{proof}

Since $\ell \le n/C$ and ${3r^{1/4}} \le (n-\ell-2)$, we have an analogue
of Corollary~\ref{cor:u1} for bounding the value of $t^{v}_1$.
\begin{corollary}
\label{cor:v1}
For every $B_1 \in \cB_1$,
\[
t^{v}_1  \le 3(n-\ell-1)kr \le \frac{3n(n-\ell-1)r^{1/4}}C \le \frac 1{C-1} \cdot \frac1{\prob{B_1}}.
\]
\end{corollary}

\begin{claim}
\label{cl2u23}
For every $u \in S$, $t^{u_2}_2(u) + t^{u_3}_2(u) \le 2 \ell \Delta_S (n-\ell)k$.
\end{claim}
\begin{proof}
This time, we show that both the value of $t^{u_2}_2(u)$ and the value of
$t^{u_3}_2(u)$ are at most $ \ell \Delta_S (n-\ell)k$.

If $u=u_2$, then there are at most $\ell$ choices for the vertex $u_1 \in L$
and at most $(\Delta_S-1)$ choices for the vertex $u_3\in S$. If $u=u_3$, then
the vertex $u_1 \in L$ can be chosen in at most $\ell$ ways and the vertex $u_2$ in at most $\Delta_S$ ways.
Next, there are $(n-\ell)$ choices for the vertex $v_2$.
Since the vertex $v_1 \in P$ is determined by the choice of the map
$f_1$, there are at most $k$ choices for the vertex $v_3 \in Q$.
\end{proof}

The inequality~(\ref{eq:maxdegS}) implies that $\Delta_S \le r^{3/8}+1$, which is at most $2r^{3/8}$.
Since $\ell = 2r^{1/4}$, we yield our next corollary.
\begin{corollary}
\label{cor:u2}
For every $B_2 \in \cB_2$,
\[
t^{u}_2 
\le 8r^{5/8} (n-\ell)k \le \frac {8n(n-\ell)}{Cr^{1/8}}
\le \frac {8n}{C(n-\ell-1)} \cdot\frac1{\prob{B_2}}
\le \frac8{C-2}\cdot\frac1{\prob{B_2}}
.\]
\end{corollary}

\begin{claim}
\label{cl2v2}
For every $v\in Q$, $t^{v_2}_2(v) \le \ell k \sqrt{2 \Delta_G r}$.
\end{claim}
\begin{proof}
There at most $\ell$ choices for the vertex $v_1 \in P$
and then at most $k$ choices for the vertex $v_3 \in Q$. 
Since the vertex $u_1 \in L$ is determined by $f_1$, Lemma~\ref{lem:leafcher}
implies that the set of two vertices $\{u_2,u_3\} \subseteq S$ can be chosen in at most $\sqrt{2 \Delta_G r}$ ways.
\end{proof}

\begin{claim}
\label{cl2v3}
For every $v \in Q$, $t^{v_3}_2(v) \le (n-\ell-1)k \sqrt{2 \Delta_G r}$.
\end{claim}
\begin{proof}
The vertex $v_2\in Q$ can be chosen in $(n-\ell-1)$ ways and the vertex 
$v_1 \in P$ in at most $k$ ways.
Then as in the previous claim, there are at most $\sqrt{2\Delta_G r}$ choices for
$\{u_2,u_3\} \subseteq S$.
\end{proof}

The choice of the parameters yields that $\ell \le (n-\ell-1)$ and $\sqrt{2\Delta_G r} \le 2 r^{3/4}$.
\begin{corollary}
\label{cor:v2}
for every $B_2 \in \cB_2$,
\[
 t^{v}_2 \le 4(n-\ell-1)k r^{3/4} \le \frac{4n}{C(n-\ell)} \cdot\frac1{\prob{B_2}}
\le
\frac 4{C-1} \cdot\frac1{\prob{B_2}}
.\]
\end{corollary}

\begin{claim}
\label{cl3u13}
For every $u \in S$, $t^{u_1}_3(u) + t^{u_3}_3(u) \le \ell (\Delta_G-1) (n-\ell) k$.
\end{claim}
\begin{proof}
First choose the vertex $u_2 \in L$; there are at most $\ell$ choices for that. The remaining
vertex in $G$, i.e., the vertex from $\{u_1,u_3\} \setminus \{u\}$, can be chosen in at most $(\Delta_G-1)$ ways.
Next, the vertex $v_2 \in P$ is given by $f_1(u_2)$. There are $(n-\ell)$ choices for $v_1 \in Q$,
and finally, at most $k$ choices for $v_3 \in Q$.
\end{proof}

\begin{claim}
\label{cl3v13}
For every $v\in Q$, $t^{v_1}_3(v) + t^{v_3}_3(v) \le 4\sqrt{r\ell}\cdot(\Delta_G -1)k$.
\end{claim}
\begin{proof}
We show that both $t^{v_1}_3(v)$ and $t^{v_3}_3(v)$
are at most $2\sqrt{r\ell}\cdot(\Delta_G -1)k$.
Suppose $v=v_1$ (the case $v=v_3$ is symmetric).
First observe since $r\ge16$, it holds that
$8r^{1/4} = 4\ell \le 2\sqrt{r\ell} = \sqrt{8}\cdot r^{5/8}$. Therefore,
Lemma~\ref{lem:edges} applies with $T:=L$ and yields that a pair of vertices $u_1
\in S$ and $u_2 \in L$ which is connected by an edge can be chosen in at most
$2 \sqrt{r \ell}$ ways. After the vertices $u_1$ and $u_2$ are chosen, there are at
most $(\Delta_G-1)$ choices for the vertex $u_3 \in S$. Since $v_2 = f_1(u_2)$,
the vertex $v_3 \in Q$ can be chosen in at most $k$ ways.
\end{proof}

Since $(\Delta_G-1)^2 \le 2r$, we conclude the following corollary.
\begin{corollary}
\label{cor:uv3}
For every $B_3 \in \cB_3$,
\[
t^{u}_3 \le 2\sqrt{2} \cdot r^{3/4}(n-\ell)k \le \frac{3n(n-\ell)}C \le \frac3{C-2}\cdot\frac1{\prob{B_3}}
\]
and
\[
t^{v}_3  \le 8r^{9/8} \cdot k = \frac{8n r^{3/8}}C \le \frac1{C-1} \cdot\frac1{\prob{B_3}}
.\]
\end{corollary}
\noindent Note the last inequality holds since $8r^{3/8} \le (n-\ell-1)$.

\begin{claim}
\label{cl4u3}
For every $u\in S$, $t^{u_3}_4(u) \le \ell(\ell-1)k$.
\end{claim}
\begin{proof}
There are at most $\ell$ choices for the vertex $u_2 \in L$ and at most
$(\ell-1)$ choices for the vertex $u_1 \in L$.
Since the vertices $\{v_1,v_2\} \subseteq P$ are determined
by $f_1$, there are at most $k$ choices for the vertex $v_3 \in Q$.
\end{proof}

\begin{claim}
\label{cl4v3}
For every $v\in Q$, $t^{v_3}_4(v) \le 2\sqrt{r\ell} \cdot k$.
\end{claim}
\begin{proof}
By Lemma~\ref{lem:edges} applied with $T:=L$, there are at most $2\sqrt{r\ell}$ choices for 
the edge $u_2u_3$ so that $u_2\in L$ and $u_3 \in S$. By definition, $v_2=f_1(u_2)$, hence the vertex $v_1 \in P$
can be chosen in at most $k$ ways.
Since $f_1$ is a bijection, the choice of $v_1$ uniquely determines the vertex $u_1$. 
\end{proof}

This time, we conclude the following.
\begin{corollary}
\label{cor:uv4}
For every $B_4\in\cB_4$,
\[
t^{u}_4 \le
4\sqrt{r}\cdot k 
\le \frac{4n}{Cr^{1/4}(n-\ell)} \cdot\frac1{\prob{B_4}}
\le \frac4{C-1}\cdot\frac1{\prob{B_4}}
\]
and
\[
t^{v}_4 
\le 2\sqrt{2}\cdot r^{5/8}\cdot k
\le \frac{3n}{Cr^{1/8}} \le \frac{3n}{C(n-\ell)} \cdot\frac1{\prob{B_4}}
\le \frac{3}{C-1} \cdot\frac1{\prob{B_4}}
.
\]
\end{corollary}

\begin{claim}
\label{cl5u2}
For every $u\in S$, $t^{u_2}_5(u) \le 2.5 \ell (\ell -1) \cdot kr^{1/4}$.
\end{claim}
\begin{proof}
There are at most $\ell \choose 2$ ways how to choose the set $\{u_1,u_3\} \subseteq L$.
This also determines the vertices $v_1 = f_1(u_1)$ and $v_3=f_1(u_3)$.
By the choice of the set $P$, there are at most $5 kr^{1/4}$ possibilities for the vertex $v_2 \in Q$.
\end{proof}

\begin{claim}
\label{cl5v2}
For every $v\in Q$, $t^{v_2}_5(v) \le 2\sqrt{r\ell} \cdot k$.
\end{claim}
\begin{proof}
First, Lemma~\ref{lem:edges} yields that there are at most $2\sqrt{r\ell}$ choices for 
the edge $u_1u_2$ with $u_1\in L$ and $u_2 \in S$. This also determines the vertex
$v_1 = f_1(u_1)$. Finally, the vertex $v_3 \in P$ can be then chosen in at most $k$ ways,
which uniquely determines the vertex $u_3 \in L$.
\end{proof}

Our final corollary is the following.
\begin{corollary}
\label{cor:uv5}
For every $B_5\in\cB_5$,
\[t^{u}_5 \le 10r^{3/4}k \le \frac{10n}{C(n-\ell)}\cdot \frac1{\prob{B_5}}
\le\frac{10}{C-1}\cdot \frac1{\prob{B_5}}
\]
and
\[
t^{v}_5 \le 2\sqrt{2} \cdot r^{5/8} \cdot k \le \frac{3n}{Cr^{1/8}}
\le \frac{3}{C-1} \cdot \frac1{\prob{B_5}}
.
\]
\end{corollary}

Corollaries~\ref{cor:u1}-\ref{cor:uv5} 
imply that 
\[
 \sum\limits_{\substack{B' \in \cB: \\ B\textrm{ and } B'\\ \cS\textrm{-intersect}}} \prob{B'}
 \le 3 \cdot \frac{25}{2C-4} + 3\cdot \frac{26}{C-1} 
 \qquad \textrm{for every } B \in \cB\mbox{.}
\]
If $C = \constSh$, then the sum above is equal to $\expShLLL < 1/4$.
Therefore, all the conditions in~(\ref{eq:proper:localcond}) are satisfied and the proof is now finished.
\end{proof}


We continue our exposition with a proof of Theorem~\ref{thm:rbshearer}, which seeks 
rainbow copies of graphs $G$ with few cherries in globally bounded colorings $c$ of $K_n$.
This time, our task is to find such a copy of $G$ in $c$ that does contain neither
a monochromatic cherry, nor a monochromatic pair of disjoint edges.
Since a globally $k$-bounded coloring is also locally $k$-bounded, it is enough
to modify the proof of Theorem~\ref{thm:shearer} by adding to the set of bad events
those that take care of all the monochromatic pairs of disjoint edges. 
As it turned out, this changes the upper bound on $k$ only by a constant factor.


\begin{proof}[Proof of Theorem~\ref{thm:rbshearer}]
Most of the proof goes along the same lines as the proof of Theorem~\ref{thm:shearer}.
Let $C:=\constShRB$ and $k:=\frac{n}{C r^{3/4}}$. 
Again, if $n < 2C$ or $r > (n/2C)^{4/3}$, the statement of the theorem is trivial.
We may assume $r\ge16$ since for $r \le 15$ the statement follows from Theorem~\ref{thm:bkp:2} (note
that $n \ge 100$, $C=42\cdot6^2$, and if $r \le 15$, then the maximum degree of $G$ is at most $6$).
Furthermore, let $V(G) = V(K_n) = [n]$, and assume the vertices of $V(G)$ are in
descending order according to their degrees.
Lemma~\ref{lem:edges} and the fact that $r\le n^{4/3}$ imply that 
$e(G) \le nr^{1/8}$.

As in the proof of Theorem~\ref{thm:shearer}, let $\Delta_G$ be the maximum
degree of $G$.  It follows that $\Delta_G \le 2\sqrt{r}$.  Let $P \subseteq
V(K_n)$ be the rainbow complete subgraph of $K_n$ of size $\ell :=
2r^{1/4}$ given by Lemma~\ref{lem:step1} for $c$, $r$ and $k$. We define $Q:=V(K_n)
\setminus P$.  On the other hand, let $L$ be the set of the first $\ell$
vertices of $G$, $S := V(G) \setminus L$ and $\Delta_S := \max_{u \in S}
\deg_G(u)$.  It holds that $\Delta_S \le 2r^{3/8}$.

The way how we find a rainbow copy of $G$ in a globally $k$-bounded coloring is
analogous to the way we have found a properly colored copy of $G$ in
a locally $k$-bounded coloring.
First, let $f_1: L \to P$ be an arbitrary bijection
and $f_2: S \to Q$ be a bijection chosen uniformly at random.
Next, let $f:=f_1 \cup f_2$ and let $f(G)$ denote the 
copy of $G$ in $K_n$ given by $f$. Our aim is to show
that Theorem~\ref{thm:luszekely} and Lemma~\ref{lem:wLLL} yield
that with a non-zero probability $f(G)$ is rainbow.

Recall from the proof of Theorem~\ref{thm:shearer} that $u_1\mathyp u_2\mathyp
u_3$ denotes a cherry in $G$ with middle vertex $u_2$ and endpoints
$u_1,u_3$ such that $u_1 < u_3$, and $\cR(G)$ is the set of all such cherries in $G$.
Also recall that for $v_1, v_2, v_3 \in V(K_n)$, the triple $[v_1 v_2 v_3]$ is
$c$-monochromatic if $c(v_1 v_2) = c(v_2 v_3)$ and $\cC(c)$ is the set of all
$c$-monochromatic triples.

In order to show that $f(G)$ is not only properly colored but rainbow,
apart from controlling the cherries we also need to guarantee there are no two disjoint
edges of the same color. This motivates the following definitions.
We write $(u_1u_2)(u_3u_4)$ to denote two disjoint
edges $u_1u_2 \in E(G)$ and $u_3u_4 \in E(G)$ such that $u_1 < u_2$, $u_3<u_4$ and $u_1 < u_3$.
Let $\cR'(G)$ be the set of all such pairs of disjoint edges $(u_1u_2)(u_3u_4)$ in $G$.
Analogously, for every $v_1, v_2, v_3, v_4 \in V(K_n)$, the quadruple $[v_1v_2v_3v_4]$ is
$c$-monochromatic if $c(v_1v_2) = c(v_3v_4)$, and we denote the set of all
$c$-monochromatic quadruples by $\cC'(G)$.

This time, our aim is to show that with positive probability the bijection $f$
is such that for every cherry $u_1\mathyp u_2\mathyp u_3 \in \cR(G)$ 
it holds that $[f(u_1)f(u_2)f(u_3)] \notin \cC(c)$,
and for every $(u_1u_2)(u_3u_4) \in \cR'(G)$
it holds that $[f(u_1)f(u_2)f(u_3)f(u_4)] \notin \cC'(c)$.
The choice of $f_1$ implies that we need to check only
the cherries $u_1\mathyp u_2 \mathyp u_3$ and the disjoint pairs of edges $(u_1u_2)(u_3u_4)$ 
that satisfy $\{u_1,u_2,u_3\} \cap S \neq \emptyset$ and $\{u_1,u_2,u_3,u_4\} \cap S \neq \emptyset$, respectively.

As in the proof of Theorem~\ref{thm:shearer}, for $u_1\mathyp u_2 \mathyp u_3 \in \cR(G)$ and $[v_1v_2v_3]\in \cC(c)$,
we denote the event $\bigwedge\limits_{i\in[3]} [f(u_i)=v_i]$
by $B^{[v_1v_2v_3]}_{u_1\mathyp u_2 \mathyp u_3}$.
We define $\cB$ to be the set of all events $B^{[v_1v_2v_3]}_{u_1\mathyp u_2 \mathyp u_3}$
such that
\begin{itemize}
\item $u_1\mathyp u_2 \mathyp u_3 \in \cR(G)$ and $[v_1v_2v_3]\in \cC(c)$,
\item $\left\{u_1,u_2,u_3\right\} \cap S \neq \emptyset$,
\item $\forall i \in \{1,2,3\}: u_i \in S \iff v_i \in Q$, and
\item $\forall i \in \{1,2,3\}: u_i \in L \Longrightarrow f_1(u_i) = v_i$.
\end{itemize}
Similarly, for $(u_1u_2)(u_3u_4) \in \cR'(G)$ and $[v_1v_2v_3v_4]\in \cC'(c)$,
let $B^{[v_1v_2v_3v_4]}_{(u_1u_2)(u_3u_4)}$
be the event $\bigwedge\limits_{i\in[4]} [f(u_i)=v_i]$.
Finally, let $\cB'$ be the set of all events $B^{[v_1v_2v_3v_4]}_{(u_1u_2)(u_3u_4)}$ such that
\begin{itemize}
\item $(u_1u_2)(u_3u_4) \in \cR'(G)$ and $[v_1v_2v_3v_4]\in \cC'(c)$,
\item $\left\{u_1,u_2,u_3,u_4\right\} \cap S \neq \emptyset$,
\item $\forall i \in \{1,2,3,4\}: u_i \in S \iff v_i \in Q$, and
\item $\forall i \in \{1,2,3,4\}: u_i \in L \Longrightarrow f_1(u_i) = v_i$.
\end{itemize}

Since the globally $k$-bounded coloring $c$ is indeed also locally $k$-bounded,
Claims~\ref{cl1u13}-\ref{cl5v2} from the proof of Theorem~\ref{thm:shearer}
apply again.
In order to upper bound the number of events $B'\in \cB$ that
intersect a given event  $B^{[v_1v_2v_3]}_{u_1\mathyp u_2 \mathyp u_3} \in \cB$
or $B^{[v_4v_5v_6v_7]}_{(u_4u_5)(u_6u_7)} \in \cB'$,
it is enough to apply these claims for vertices $u\in\{u_1,u_2,u_3\} \cap S$ and $v\in\{v_1,v_2,v_3\} \cap Q$,
or $u\in\{u_4,u_5,u_6,u_7\} \cap S$ and $v\in\{v_4,v_5,v_6,v_7\} \cap Q$, respectively.
In all the possible cases, there are at most $4$~choices for such a vertex.
Therefore, Corollaries~\ref{cor:u1}-\ref{cor:uv5} yield that

\begin{equation}
 \sum\limits_{\substack{B' \in \cB: \\ B\textrm{ and } B'\\ \cS\textrm{-intersect}}} \prob{B'}
 \le 4 \cdot \frac{25}{2C-4} + 4\cdot \frac{26}{C-1} 
 \qquad \textrm{for every } B \in \cB \cup \cB'\mbox{.}
\label{eq:rbcherrybound}
\end{equation}

It remains to analyze how many events from $\cB'$
a fixed event $B \in \cB \cup \cB'$ can $\cS$-intersect.
We start with splitting the events $B^{[v_1v_2v_3v_4]}_{u_1\mathyp u_2, u_3\mathyp u_4} \in \cB'$ into 
five classes $\cB_6,\dots,\cB_{10}$ based on how their sets $\{u_1,u_2,u_3, u_4\}$ intersect the set $S$:
\begin{itemize}
\item If $\{u_1,u_2,u_3,u_4\} \subseteq S$, then $B^{[v_1v_2v_3v_4]}_{u_1\mathyp u_2, u_3\mathyp u_4} \in \cB_6$.
\item If $\{u_2, u_3,u_4\} \subseteq S$ and $u_1 \in L$, then $B^{[v_1v_2v_3v_4]}_{u_1\mathyp u_2, u_3\mathyp u_4} \in \cB_7$.
\item If $\{u_3,u_4\} \subseteq S$ and $\{u_1,u_2\} \subseteq L$, then $B^{[v_1v_2v_3v_4]}_{u_1\mathyp u_2, u_3\mathyp u_4} \in \cB_8$.
\item If $\{u_2,u_4\} \subseteq S$ and $\{u_1,u_3\} \subseteq L$, then $B^{[v_1v_2v_3v_4]}_{u_1\mathyp u_2, u_3\mathyp u_4} \in \cB_9$.
\item If $|\{u_2,u_4\} \cap S|=1$ and $\{u_1,u_3\} \subseteq L$, then $B^{[v_1v_2v_3v_4]}_{u_1\mathyp u_2, u_3\mathyp u_4} \in \cB_{10}$;
\end{itemize}
see also Figure~\ref{fig:events_edges}.
The fact that the classes $\cB_6,\dots,\cB_{10}$ split the whole
set $\cB'$ follows 
because if $u_i \in L$ for some $i\in [4]$, then $u_1 \in L$,
and also if $u_4 \in L$, then $u_3 \in L$.
It holds that
\begin{align*}
& \prob{B}=\frac1{(n-\ell)(n-\ell-1)(n-\ell-2)(n-\ell-3)}\quad & \textrm{for any } & B\in\cB_6,\\
& \prob{B}=\frac1{(n-\ell)(n-\ell-1)(n-\ell-2)}\quad & \textrm{for any } & B\in\cB_7,\\
& \prob{B}=\frac1{(n-\ell)(n-\ell-1)}\quad & \textrm{for any } & B\in\cB_8 \cup \cB_9, \textrm{ and}\\
& \prob{B}=\frac1{(n-\ell)}\quad & \textrm{for any } & B\in\cB_{10}.
\end{align*}

\begin{figure}[htp]
\begin{center}
\medskip
\foreach \n in {6,...,9}{ \includegraphics[scale=0.8,page=\n]{fig-lll_simple}\hfill}
{ \includegraphics[scale=0.8,page=10]{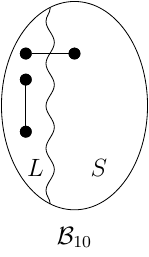}}
\end{center}
\caption{The event classes $\cB_6,\dots,\cB_{10}$.}
\label{fig:events_edges}
\end{figure}

For every vertex $u \in S$ and two integers $i \in \{6,\dots,10\}$ and $j\in[4]$, let $t^{u_j}_i (u)$ be the number of events
$B^{[v_1v_2v_3v_4]}_{(u_1u_2)(u_3u_4)} \in \cB_i$ such that $u=u_j$. It immediately follows that all
$t^{u_1}_7(u)$, $t^{u_1}_8(u)$, $t^{u_2}_8(u)$, $t^{u_1}_9(u)$, $t^{u_3}_9(u)$, $t^{u_1}_{10}(u)$ and $t^{u_3}_{10}(u)$
are zero for every vertex $u \in S$.
Similarly, for every vertex $v \in Q$ and integers $i \in \{6,\dots,10\}$ and $j\in[4]$,
let $t^{v_j}_i (v)$ be the number of events $B^{[v_1v_2v_3v_4]}_{(u_1u_2)(u_3u_4)} \in \cB_i$ such that $v=v_j$.
Analogously to the previous case, the values of
$t^{v_1}_7(v)$, $t^{v_1}_8(v)$, $t^{v_2}_8(v)$, $t^{v_1}_9(v)$, $t^{v_3}_9(v)$, $t^{v_1}_{10}(u)$ and $t^{v_3}_{10}(v)$
are equal to~$0$ for all $v \in Q$.
Therefore, for every $B \in \cB \cup \cB'$ it holds that
\[
\sum\limits_{\substack{B' \in \cB': \\ B\textrm{ and } B'\\ \cS\textrm{-intersect}}} \prob{B'} \; \le \;
\sum\limits_{i=6}^{10}
{\prob{B'_i}}\cdot 4\left({t^u_i + t^v_i}\right)
\mbox{\,,}
\]
where $B'_i \in \cB_i$ for $i \in \{6,\dots,10\}$.

In order to finish the proof, we perform similar calculations as we did in the proof of Theorem~\ref{thm:shearer}
in order to give upper bounds on $t^{u_j}_i$ and $t^{v_j}_i$, where $i\in\{6,\dots,10\}$ and $j\in[4]$.
\begin{claim}
\label{cl6u1234}
For every $u \in S$, 
$\displaystyle\sum\limits_{j \in [4]} t^{u_j}_6(u) \le 2e(G) \cdot \Delta_S (n-\ell)(n-\ell-1) \cdot k$.
\end{claim}
\begin{proof}
A neighbor $u' \in S$ of $u$ can be chosen in at most $\Delta_S$ ways,
and then there are at most $e(G)$ choices for the edge $u''u'''$ disjoint from $uu'$.
Note that the relative order between $u,u',u''$ and $u'''$ uniquely determines how
these vertices correspond to $u_1,u_2,u_3$ and $u_4$. Next, the vertices $v_1$ and $v_2$
can be chosen in $(n-\ell)(n-\ell-1)$ ways. Finally, there are at most $k$ edges $v'v''$ 
with color $c(v_1v_2)$ and then we only need to decide whether $v_3=v'$ and $v_4=v''$,
or the other way around.
\end{proof}

\begin{claim}
\label{cl6v1234}
For every $v \in Q$, 
$\displaystyle\sum\limits_{j\in[4]} t^{v_j}_6(v) \le 4e(G)^2 (n-\ell-1)k$.
\end{claim}
\begin{proof}
This time we show that $t^{v_j}_6(v)$ is at most ${e(G)^2} (n-\ell-1)k$
for every $j\in[4]$. Without loss of generality, $v=v_1$. There are $(n-\ell-1)$
choices for $v_2$ and then, as in the previous claim, at most $2k$ choices for $v_3$ and $v_4$.
On the other hand, the total number of choices for the vertices $u_1,u_2,u_3$ and $u_4$
is at most $e(G) \choose 2$.
\end{proof}

The estimates $e(G) \le nr^{1/8}$ and $\Delta_S \le 2r^{3/8}$ yields the following.
\begin{corollary}
\label{cor:uv6}
For every $B_6 \in \cB_6$,
\[ 
t^u_6
\le 4 nr^{1/2} \cdot (n-\ell)(n-\ell-1)k \le
\frac{4n^2(n-\ell)(n-\ell-1)}{Cr^{1/4}} \le \frac 4{C-5} \cdot \frac 1{\prob{B_6}}\]
and
\[ 
t^v_6 \le 4n^2 r^{1/4} \cdot (n-\ell-1) k
\le \frac{4n^3(n-\ell-1)}{Cr^{1/2}} \le \frac{4}{C-6} \cdot \frac1{\prob{B_6}}.\]
\end{corollary}

\begin{claim}
\label{cl7u2}
For every $u \in S$, 
$t^{u_2}_7(u) \le 2e(G) \cdot \ell (n-\ell) k$.
\end{claim}
\begin{proof}
The vertex $u_1 \in L$ can be chosen in at most $\ell$ ways, the vertices $u_3$ and $u_4$
in at most $e(G)$ ways, and the vertex $v_2 \in Q$ in $(n-\ell)$ ways.
Since the vertex $v_1 = f(u_1)$, there are at most $2k$ choices for the vertices $v_3$ and $v_4$.
\end{proof}

\begin{claim}
\label{cl7u34}
For every $u \in S$, 
$t^{u_3}_7(u) + t^{u_4}_7(u) \le 4 \Delta_S \sqrt{r\ell} \cdot (n-\ell)k$.
\end{claim}
\begin{proof}
There are at most $\Delta_S$ choices for the vertex $u' \in \{u_3,u_4\}\setminus \{u\}$
and, by Lemma~\ref{lem:edges}, at most $2\sqrt{r\ell}$ choices for $u_1 \in L$ and $u_2 \in S$.
The total number of choices for $v_1,v_2,v_3$ and $v_4$ is at most $2k(n-\ell)$.
\end{proof}

Since $8\sqrt{2} \cdot r\le nr^{3/8}$, we conclude the following.
\begin{corollary}
\label{cor:u7}
For every $B_7 \in \cB_7$,
\[t^{u}_7 \le (n-\ell)k \cdot (4nr^{3/8} + 8\sqrt{2} \cdot r)
\le \frac{5n^2(n-\ell)}{Cr^{3/8}} \le \frac{5}{C-4}
\cdot \frac1{\prob{B_7}}.\]
\end{corollary}

\begin{claim}
\label{cl7v2}
For every $v \in Q$, 
$t^{v_2}_7(v) \le 4\sqrt{r\ell} \cdot e(G)k$.
\end{claim}
\begin{proof}
There are at most $2\sqrt{r\ell}$ choices for the vertices $u_1 \in L$ and $u_2 \in S$.
This determines the vertex $v_1 \in P$ and hence the vertices $v_3$ and $v_4$ can be chosen
in at most $2k$ ways. Finally, the remaining vertices $u_3$ and $u_4$ are determined by choosing an edge of $G$
that has both endpoints in $S$.
\end{proof}

\begin{claim}
\label{cl7v34}
For every $v \in Q$, 
$t^{v_3}_7(v) + t^{v_4}_7(v) \le 2 e(G) \Delta_G(n-\ell-1)k$.
\end{claim}
\begin{proof}
By symmetry, it is enough to show that $t^{v_3}_7(v) \le e(G)\Delta_G (n-\ell-1)k$.
There are $(n-\ell-1)$ choices for the vertex $v_4 \in Q$, then at most $k$ choices
for $v_1 \in P$ and $v_2 \in Q$, and since $u_1 = f_1^{-1}(v_1)$, at most $\Delta_G$ choices
for $u_2$. As in the previous claims, the vertices $u_3$ and $u_4$ can be chosen in at most
$e(G)$ ways.
\end{proof}

Recall that $\Delta_G \le 2\sqrt{r}$. The counterpart of Corollary~\ref{cor:u7} is the following.
\begin{corollary}
\label{cor:v7}
For every $B_7 \in \cB_7$,
\[t^{v}_7 \le
4\sqrt{2} \cdot nr^{3/4}k + 4nr^{5/8} (n-\ell-1)k
\le \frac{5n^2(n-\ell-1)}{C} \le \frac5{C-3} 
\cdot \frac1{\prob{B_7}}.\]
\end{corollary}

\begin{claim}
\label{cl8u34}
For every $u \in S$, 
$t^{u_3}_8(u) + t^{u_4}_8(u) \le \Delta_S \ell^2 \cdot k$.
\end{claim}
\begin{proof}
We first choose the vertices $u_1 \in L$ and $u_2 \in L$
such that $u_1 < u_2$.
This can be done in at most $\ell \choose 2$ ways, and it determines the vertices $v_1 \in P$ and $v_2 \in P$.
After that, there are at most $2k$ choices for the vertices $v_3 \in Q$ and $v_4 \in Q$.
The only remaining vertex we need to choose is a neighbor of $u$, and there are at most $\Delta_S$ ways to do that.
\end{proof}

\begin{claim}
\label{cl8v34}
For every $v \in Q$, 
$t^{v_3}_8(v) + t^{v_4}_8(v) \le e(G) \ell^2 \cdot k$.
\end{claim}
\begin{proof}
Analogously to the proofs of Claims~\ref{cl6v1234} and~\ref{cl7v34}, it is enough to show that $t^{v_3}_8(v) \le e(G) {\ell\choose2} k$.
We can choose the vertices $u_1 \in L$ and $u_2 \in L$ such that $u_1 < u_2$ in at most~$\ell \choose 2$ ways,
then there at most $k$ choices for the vertex $v_4 \in Q$, and finally at most~$e(G)$ choices
for the vertices $u_3$ and $u_4$.
\end{proof}

Claims~\ref{cl8u34} and~\ref{cl8v34} yields our next corollary.
\begin{corollary}
\label{cor:uv8}
For every $B_8 \in \cB_8$,
\[ 
t^{u}_8 \le 8r^{7/8} \cdot k \le \frac{8nr^{1/8}}{C} \le \frac1{C-1} \cdot \frac1{\prob{B_8}}
\]
and
\[
t^{v}_8 \le 4nr^{5/8} \cdot k \le \frac{4n^2}C \le \frac4{C-3} \cdot \frac1{\prob{B_8}}
.\]
\end{corollary}

\begin{claim}
\label{cl9u24}
For every $u \in S$, 
$t^{u_2}_9(u) + t^{u_4}_9(u) \le 2\sqrt{r\ell}\cdot (n-\ell)k$.
\end{claim}
\begin{proof}
We start by choosing an adjacent pair of vertices $u'' \in S$ and $u''' \in L$.
Lemma~\ref{lem:edges} implies this can be done in at most $2\sqrt{r\ell}$ ways.
Then, in $(n-\ell)$ ways, we choose the vertex $v'' \in Q$ which will be the image
of $u''$. The vertices $v \in Q$ and $v' \in P$ can be then chosen in at most $k$ ways,
which also uniquely determines the vertex $u' \in L$. The relative order of $u'$ and $u'''$
determines the correspondence between $u,u',u'',u'''$ and $u_1,u_2,u_3,u_4$, which
gives also the correspondence between $v,v',v'',v'''$ and $v_1,v_2,v_3,v_4$.
\end{proof}

\begin{claim}
\label{cl9v24}
For every $v \in Q$, 
$t^{v_2}_9(v) + t^{v_4}_9(v) \le 2 \left(\Delta_G\right)^2 \ell k$.
\end{claim}
\begin{proof}
As usual, it is enough to show that $t^{v_2}_9(v) \le \left(\Delta_G\right)^2 \ell k$.
There are at most $\ell$ choices for the vertex $v_1 \in P$ and after that at most $k$
choices for $v_3 \in P$ and $v_4 \in Q$. Since the vertices $u_2 \in S$ and $u_4 \in S$ 
are neighbors of $u_1 = f_1^{-1}(v_1)$ and $u_3 = f_1^{-1}(v_3)$, respectively,
each of them can be chosen in at most $\Delta_G$ ways.
\end{proof}

We use the estimate $16 \sqrt{r} \le n-\ell-1$ to obtain the following corollary.
\begin{corollary}
\label{cor:uv9}
For every $B_9 \in \cB_9$,
\[ 
t^{u}_9 \le 2\sqrt2 \cdot r^{5/8}(n-\ell)k \le \frac{3n(n-\ell)}{C}  \le \frac3{C-2}
\cdot \frac1{\prob{B_9}}
\]
and
\[
t^{v}_9 \le 16 \cdot r^{5/4}k \le \frac{16 \sqrt{r} \cdot n}{C} \le \frac1{C-1}
\cdot \frac1{\prob{B_9}}
.
\]
\end{corollary}

\begin{claim}
\label{cl10u24}
For every $u \in S$, 
$t^{u_2}_{10}(u) + t^{u_4}_{10}(u) \le \ell^2k$.
\end{claim}
\begin{proof}
By symmetry, it is enough to show that $t^{u_2}_{10}(u) \le \frac12 \ell^2 k$.
Indeed, we choose the vertices $u_3 \in L$ and $u_4 \in L$ in $\ell \choose 2$
ways, and after that there are at most $k$ choices for the vertices $v_1 \in P$ and $v_2 \in Q$.
\end{proof}

\begin{claim}
\label{cl10v24}
For every $v \in Q$, 
$t^{v_2}_{10}(v) + t^{v_4}_{10}(v) \le 2 \ell \Delta_G k$.
\end{claim}
\begin{proof}
Analogously to the previous claim, we only show that 
$t^{v_2}_{10}(v)$ is at most $\ell \Delta_G k$. A symmetric
reasoning then yields the same upper bound also holds for~$t^{v_4}_{10}(v)$.

There are at most $\ell$ choices for the vertex $v_1 \in P$,
then at most $k$ choices for the vertices $v_3$ and $v_4$ (note that the ordering of $u_3 < u_4$ defines an ordering of $v_3$ and $v_4$).
Finally, at most $\Delta_G$ choices for the neighbor of $u_1 \in L$, i.e.,
the vertex $u_2 \in S$.
\end{proof}

Here comes the last corollary.
\begin{corollary}
\label{cor:uv10}
For every $B_{10} \in \cB_{10}$,
\[ 
t^{u}_{10} \le 4\sqrt{r} \cdot k \le \frac{4n}C \le \frac{4}{C-1}
\cdot \frac1{\prob{B_{10}}}
\]
and
\[
t^{v}_{10} \le  8r^{3/4} k \le \frac{8n}C \le \frac{8}{C-1}
\cdot \frac1{\prob{B_{10}}}
.
\]
\end{corollary}

Corollaries~\ref{cor:uv6}-\ref{cor:uv10} 
imply that for every $B \in \cB \cup \cB'$, it holds that
\[
 \sum\limits_{\substack{B' \in \cB': \\ B\textrm{ and } B'\,\cS\textrm{-intersect}}} \prob{B'}
 \le  \frac{4\cdot 14}{C-1} + \frac{4\cdot3}{C-2} + \frac{4\cdot9}{C-3} + \frac{4\cdot5}{C-4} + \frac{4\cdot4}{C-5} + \cdot \frac{4\cdot4}{C-6}  
.
\]
The last upper bound together with (\ref{eq:rbcherrybound}) and our choice of the constant $C$ imply that $f(G)$ is rainbow with a non-zero probability.
\end{proof}


\section{Lower bounds}
\label{sec:lbounds}
In this section, we present three constructions of bounded colorings $c$
and graphs $G$ with either small number of cherries or small maximum degree, 
which provides the matching lower bounds for Theorems~\ref{thm:shearer},~\ref{thm:bkp:1},~\ref{thm:bkp:2} and~\ref{thm:rbshearer}.
We start with constructing an edge-coloring of $K_n$ that does not contain properly colored spanning trees
of radius two.
\begin{lemma}
\label{lem:rad2tree}
For every integer $n$,
there exists a locally $3$-bounded edge-coloring $c$ of $K_{3n}$ such that
$c$ contains no properly edge-colored spanning tree of radius two.
Moreover, the coloring $c$ is globally $9$-bounded.
\end{lemma}
\begin{proof}
Split arbitrarily the vertex-set $V(K_n)$ into $n$ disjoint parts $P_1,\dots,P_{n}$, each of size $3$.
The coloring $c$ uses a palette of colors $[n] \cup {[n] \choose 2}$ 
and two vertices $x \in P_i$ and $y \in P_j$, where $i \in [n]$ and $j \in [n]$, are colored with the color $\{i,j\}$.
Note that if $i=j$, the edge $xy$ has color $\{i\}$.
It follows that the coloring $c$ is locally $3$-bounded and globally $9$-bounded.

Fix a tree $T$ of radius two 
and let $u$ be a central vertex of $T$, i.e., a vertex that has distance at most $2$ from
every $u' \in V(T)$. Suppose for contradiction that $c$ is $T$-proper. Fix a properly colored
copy of $T$ and let $v \in V(K_n)$ be the vertex corresponding to $u$. Without loss of generality,
$v \in P_1$. Let $v_2 \in P_1$ and $v_3 \in P_1$ be the other two vertices from the part
$P_1$, and $u_2 \in V(T)$ and $u_3 \in V(T)$ their corresponding vertices in $T$.
Since $T$ is properly colored, at least one of $u_2$ and $u_3$ is at distance two from $u$.
Without loss of generality, $u$ and $u_2$ have distance two in $T$, and let
$u_4$ be their (unique) common neighbor.
But then $c(vv_4)=c(v_2v_4)$,  where $v_4\in V(K_n)$ is the corresponding vertex to $u_4$, a contradiction.
\end{proof}

For an integer $m$, let $T_m$ be a tree of radius two with exactly one vertex of
degree $m^{2/3}$ that has all the neighbors of degree $m^{1/3}+1$ and they
have all their other neighbors of degree one.
Note that $T_m$ has $n := m + m^{2/3} + 1$ vertices and contains 
$\binom{m^{2/3}}2 +m^{2/3} \cdot m^{1/3} + m^{2/3} \cdot \binom{m^{1/3}}2 = m^{4/3} + (m-m^{2/3})/2 = \Theta(n^{4/3})$
cherries. Applying the previous lemma to $T_m$, we conclude
that the upper bounds on $k$ in Theorems~\ref{thm:shearer} and~\ref{thm:rbshearer} are, up to a constant factor,
best possible even when we restrict the graphs $G$ only to be trees.
\begin{corollary}
For every integer $n$, there exist an $n$-vertex tree $T$ with $\Theta(n^{4/3})$ cherries
and a locally $3$-bounded coloring $c$ of $K_n$ such that $c$ is not
$T$-proper.  Moreover, the coloring $c$ is globally $9$-bounded.
\end{corollary}

Next, consider a tree $T'_m$ of radius two with one vertex of degree $\sqrt{m}$, all its
neighbors of degree $\sqrt{m}$ and all their other neighbors of degree one. It follows
that $T'_m$ has $m+1$ vertices and maximum degree $\sqrt{m}$.
Lemma~\ref{lem:rad2tree} implies that both Theorems~\ref{thm:bkp:1}
and~\ref{thm:bkp:2} are tight in the regime $\Delta(G) = \Theta(\sqrt{n})$.

Now we present a similar type of coloring to the one from
Lemma~\ref{lem:rad2tree} which will not contain any properly colored graph of
diameter two. We will then use it to show that Theorem~\ref{thm:bkp:1} is in fact
tight, again up to a constant factor, for all values of $n$ and $\Delta$. Even
more, in this case we do not need $G$ to be spanning.  In fact $G$ can be of a
fixed order completely independent on $n$ (more precisely, our graphs $G$ will
be only of order $\Theta\left(\Delta^2\right)$).
Let us start with the following auxiliary lemma.
\begin{lemma}
\label{lem:diam2prop}
For a fixed integer $\ell \ge 3$, there exists a locally $(3n/\ell)$-bounded edge-coloring of $K_n$
such that $c$ contains no properly colored $\ell$-vertex graph of diameter two.
\end{lemma}
\begin{proof}
Split the vertex-set $V(K_n)$ into $n$ parts $P_1,\dots,P_{\ell/3}$, each of
size $3n/\ell$.  Analogously to the proof of Lemma~\ref{lem:rad2tree}, the coloring $c$
uses a palette of colors $[\ell/3] \cup {[\ell/3] \choose 2}$ and two vertices
$x\in P_i$ and $y \in P_j$ are colored with the color $\{i,j\}$.  It
holds that $c$ is locally $(3n/\ell)$-bounded.

Now let $G$ be an $\ell$-vertex graph of diameter two, and
suppose $c$ contains a properly colored copy of $G$.
By the pigeonhole principle, at least one of the parts $P_i \subseteq V(K_n)$ contains
at least three vertices of $G$. Let $u_1,u_2,u_3 \in V(G)$ be those vertices.
If there is a pair of vertices from $\{u_1,u_2,u_3\}$ that does not span an edge in $G$, then there is
no part $P_j$ for its common neighbor so that we avoid having a monochromatic path on three vertices in $c$.
But that means $\{u_1,u_2,u_3\}$ must be a triangle in $G$. Since $c(v_1v_2)=c(v_2v_3)=c(v_3v_1)$,
we conclude that $c$ does not contain a properly edge-colored copy of~$G$.
\end{proof}

We are now ready to prove Proposition~\ref{prop:constr1}.
\begin{proof}[Proof of Proposition~\ref{prop:constr1}]
Let $PG(2,q)$ be a projective plane of order $q$,
and let $G_{q}$ be the orthogonal polarity graph of $PG(2,q)$, which was introduced by Erd\H{o}s and R\'enyi in~\cite{bib:ERpolarity}.
Specifically, the vertex set of $G_q$ is the set of all points of $PG(2,q)$, where two distinct vertices $(x_1,x_2,x_3)$
and $(y_1,y_2,y_3)$ are adjacent if and only if $x_1y_1+x_2y_2+x_3y_3=0$.
It follows that $G_q$ has $\ell:=q^2 + q + 1$ vertices, maximum degree $\Delta:=q+1$, and diameter two.
Note that $\sqrt{\ell} \le \Delta \le \sqrt{1.3\ell}$.
It follows that the edge-coloring of $K_n$ from Lemma~\ref{lem:diam2prop}
is not $G_q$-proper.
\end{proof}

We finish this section with a construction of a coloring suitable for showing that also Theorem~\ref{thm:bkp:2} 
is tight, up to a constant factor, for any choice of $\Delta$.
We start with an analogue of Lemma~\ref{lem:diam2prop}.
\begin{lemma}
\label{lem:diam2rb}
Fix integers $\ell > 0$ and $n \ge 4\ell$, and let $G$ be a graph of diameter two with the additional
property that for every $v \in V(G)$, the neighborhood of $v$ does not contain an independent set of size $3$.
Then there exists a globally $(4\ell)$-bounded coloring of $K_n$
such that any rainbow copy of $G$ in $c$ contains at most $3$ vertices from
the set $\{1,\dots,4\ell\} \subseteq V(K_n)$.
\end{lemma}
\begin{proof}
Let $X := \{1,\dots,4\ell\}$. It will be enough to describe just the colors
of the edges with at least one endpoint in $X$ (the coloring of the subgraph induced by $V(K_n) \setminus X$
can be, for instance, rainbow using only colors that are disjoint from those we use in the rest of this paragraph).
The set of colors we use for edges that touch $X$ will be $[n]$.
If both $v_1 \in X$ and $v_2 \in X$, then we color $v_1v_2$ with
$\min(v_1,v_2)$. In other words, $c$ is a lexicographic coloring on the set
$X$. On the other hand, if $v_1 \in X$ and $v_2 \in V(K_n)\setminus X$, the
color of $v_1v_2$ will be $v_2$.

Suppose there is a rainbow copy of $G$ that contains $z \ge 4$ vertices from the set $X$.
Let $v_1 < v_2< \dots < v_z$ be these vertices, and let $u_1,\dots,u_z$ be their corresponding
vertices in $G$. For convenience, we also write $u_i < u_j$ if $1 \le i < j \le z$.
It follows from the definition of $c$ that at most one of the pairs $u_1u_2$ and $u_1u_3$ can form an edge of $G$.

First consider the case $u_1u_2 \in E(G)$. Let $u$ be a common neighbor of
$u_1$ and $u_3$,  and $v$ the vertex in $K_n$ corresponding to $u$. It follows that $v$
must be in $X$ (as otherwise $c(vv_1)=c(vv_3)=v$). Even more, $v$ actually must
be $v_2$ (otherwise $c(v_1v_2)=c(v_1v)$). Translated back to $G$, we conclude that $u=u_2$.
The same reasoning applied to $u_1$ and $u_4$ yields that $u_2$ is
also their common neighbor. But this is impossible, since $c(v_2v_3) =
c(v_2v_4)$.

Now suppose $u_1u_3 \in E(G)$ (and hence $u_1u_2 \notin E(G)$). Then the only common neighbor of $u_1$ and $u_2$ can be $u_3$
and, analogously, the only neighbor of $u_1$ and $u_4$ can be $u_3$. But that means
that all the vertices $u_1,u_2$ and $u_4$ are neighbors of $u_3$, hence at least one of
the three pairs from $\{u_1,u_2,u_4\}$ is an edge of $G$. Let $uu'$ such that $u<u'$ be
one such edge, and let $v\in V(K_n)$ and $v' \in V(K_n)$ be the vertices corresponding to $u$ and $u'$, respectively.
Since $v <v_3$, it follows that $c(vv') = c(vv_3) = v$, a contradiction.

Finally, consider the case when $u_1u_2 \notin E(G)$ and $u_1u_3 \notin E(G)$. Let $u\in V(G)$ be a common neighbor of $u_1$ and $u_2$
and $v$ its corresponding vertex in $K_n$. Note that $v > v_3$.
By the same reasoning as in the previous two paragraphs, $v \in X$, and $u$ is also a common neighbor of $u_1$ and $u_3$.
But then the vertices $u_1, u_2$ and $u_3$ are all neighbors of $u$ and hence they must span at least one edge in $G$.
It follows that this edge must be $u_2u_3$. But since $c(u_2u_3)=c(u_2u)=u_2$, the proof of the lemma is finished.
\end{proof}

For an integer $m$, let $H_m$ be an $m^2$-vertex graph with the
vertex set $[m] \times [m]$, where two vertices $(i,j)$ and $(i',j')$ are
adjacent if and only if $i=i'$ or $j=j'$. $H_m$ has 
maximum degree $2m-2$, diameter two, and for each $v \in V(H_m)$, the
neighborhood of $v$ induces a subgraph with independence number at most $2$.
We conclude the section by applying Lemma~\ref{lem:diam2rb} to $n$-vertex graphs that are
disjoint unions of $n/m^2$ copies of $H_m$.
\begin{proof}[Proof of Proposition~\ref{prop:constr2}]
Let $m:=\Delta/2+1$ and $G$ be an $n$-vertex graph consisting of $\ell := n/m^2$ disjoint copies of $H_m$.
Note that the maximum degree of $G$ is $\Delta$.
Next, let $c$ be the globally $\left({16n}/{\Delta^2}\right)$-bounded coloring from Lemma~\ref{lem:diam2rb} applied with $n$ and $\ell$.

Suppose $c$ is $G$-rainbow.
By the pigeonhole principle, at least one of the $\ell$ copies of $H_m$ must contain at
least $4$ vertices from the set $X:=\{1,\dots,4\ell\}$. However, Lemma~\ref{lem:diam2rb}
implies that each rainbow copy of $H_m$ can intersect $X$ in at most $3$ vertices, a contradiction.
\end{proof}

\section{Concluding remarks}
\label{sec:remarks}
In this paper we showed that any locally $k$-bounded
edge-coloring of $K_n$ with constant $k$ is $G$-proper for all $n$-vertex graphs $G$ with at most $O(n^{4/3})$ cherries.
In particular, this confirms an old conjecture of Shearer. Moreover, the bound $\Theta\left(n^{4/3}\right)$ is best
possible, even if we restrict our attention only to trees.
More generally, we proved that if $G$ is an $n$-vertex graph with $r$ cherries, 
any locally $k$-bounded edge-coloring of $K_n$ is $G$-proper for $k=O\left(\frac n{r^{3/4}}\right)$.
However, we do not know whether the dependency $k=O\left(\frac n{r^{3/4}}\right)$
for graphs $G$ with $r \ll n^{4/3}$ cherries is optimal.
Similarly, is the same dependency best possible for finding a rainbow copy of $G$
in globally $k$-bounded edge-colorings of $K_n$?

We have also observed that the dependency $k=O\left(n/\Delta^2\right)$ in Theorems~\ref{thm:bkp:1} and~\ref{thm:bkp:2} cannot be 
further improved, even in the case when $G$ is a spanning tree, e.g, consider the $\sqrt n$-ary tree of radius two. However, a simple greedy embedding 
together with the fact that trees are $1$-degenerate shows
that if $G$ is a tree on $(1-\eps)n$ vertices with maximum degree $\Delta$ and $k=\eps n/\Delta$,
then any locally $k$-bounded coloring of $E(K_n)$ is $G$-proper. This leads to a natural question
whether the bound $k=O(n/\Delta^2)$ can be improved for spanning trees with maximum degree $\Delta \ll \sqrt{n}$.

Finally, for any graph $G$ with maximum degree $\Delta$, the proofs of Theorems~\ref{thm:bkp:1} and~\ref{thm:bkp:2}
hold (with slightly worse constants in the upper bounds on~$k$) even if we replace the graph 
$K_n$ by a graph $K$ with minimum degree at least $n - O\left(\frac{n}{\Delta(G)}\right)$.
This follows simply by adding to the set of bad events in the application of local lemma 
those events, that take care of mapping an edge of $G$ to a non-edge of $K$. The corresponding
proofs are then modified analogously to the modification of the proof of Theorem~\ref{thm:shearer} in order to
establish Theorem~\ref{thm:rbshearer}.
Therefore, if $c$ is a locally (globally) bounded coloring of the edges of $K_n$ as stated in
Theorem~\ref{thm:bkp:1} (Theorem~\ref{thm:bkp:2}), we can find, by iteratively applying the previous
claim, $\Theta\left(\frac{n}{\Delta^2}\right)$ properly colored (rainbow) edge-disjoint copies
of $G$ in $c$ instead of just one.
Similarly, the proofs of Theorems~\ref{thm:shearer} and~\ref{thm:rbshearer} can be used to find 
properly colored and rainbow copies of a graph with $r$ cherries in bounded
colorings of graphs with large minimum degree.

\vspace{0.3cm}

{\bf Acknowledgments.}\,  The authors would like to thank Nina Kam\v{c}ev and
Michael Krivelevich for fruitful discussions on topics related to this project,
and the anonymous referees for their valuable comments, which greatly improved
the presentation of the results.

{

\end{document}